\documentclass[12pt]{article}%
\usepackage{amsmath}
\usepackage{amsfonts}
\usepackage{amssymb}
\usepackage{graphicx}
\usepackage{enumerate}
\numberwithin{equation}{section}
\numberwithin{figure}{section}
\numberwithin{table}{section}

\def\theequation{\arabic{section}.\arabic{equation}}
\newtheorem{theorem}{Theorem}[section]
\newtheorem{corollary}{Corollary}[section]

\newenvironment{proof}[1][Proof]{\medskip\noindent\textbf{#1.} }{ $\square$\medskip}

\begin{document}

\bigskip
\noindent
{\Large \bf Aggregate claims when their sizes and arrival times are dependent and governed by a general point process}
\vspace*{8mm}

\noindent
{\large Kristina P. Sendova and Ri\v{c}ardas Zitikis}

\bigskip
\noindent
\textit{Department of Statistical and Actuarial Sciences, University of Western Ontario, 1151 Richmond Street North,
London, Ontario, N6A 5B7, Canada}

\vspace*{8mm}
\noindent
\textbf{Abstract.}
We suggest a general method for analyzing aggregate insurance claims that arrive according to a very general point process, known in the literature as the order statistic point process, which includes as special cases the classical compound Poisson and the Sparre Andersen models, among many others. We also allow for the process to govern claim sizes via a general dependence structure that relates claim sizes to claim and/or inter-claim times. The obtained general results are supplemented with special closed-form illustrative formulas, that also exemplify the potential for future research in the area.


\newpage

\section{Introduction}
\label{sc: Introduction}

We are interested in building a general and encompassing stochastic model for  aggregate losses that insurance companies accumulate over certain periods of time. Naturally, the model is based on a claim-arrival process, which we seek to be a very general point process so that well studied models such as the compound Poisson and  Sparre Andersen models would be included. We refer to, for example, Li (2008), Ren (2008), L\'eveill\'e et al. (2010), and references therein, for results involving the aggregate (discounted) claim process under the Sparre Andersen model.

Our interest in the problem has also been motivated by other intriguing studies such as modelling hurricane losses, as initiated in the works by Garrido and Lu (2004), and Lu and Garrido (2005, 2006). These authors assume that claims arrive according to a non-homogeneous Poisson process, which can reflect the seasonality and other periodic features of claim arrivals.

In addition to the general claim arrival process, we also tackle situations when claim sizes are dependent on preceding inter-claim times, as in the work by Boudrault et al. (2006). Such extensions and generalizations facilitate modelling aggregate claims related to, for example, volcano eruptions and earthquakes, when longer time intervals between consecutive volcano eruptions or earthquakes often result in more severe damages. We tackle even more general situations when claim sizes depend on their arrival times, which allows us to incorporate features such as the seasonality of certain claim arrivals (e.g., Garrido and Lu, 2004; Lu and Garrido, 2005, 2006).

We next develop some mathematical formalism. Let $T_ 1, T_2, \dots $ denote claim arrival times, and let $X_ 1, X_2, \dots $ be the corresponding claim sizes. Let $N(t)$ denote the number of claims that  arrive up to and including time $t>0$. Hence, $N(t)$ is the largest integer such that
$0<T_1 < \cdots < T_{N(t)}\le t$. The aggregate claim size of all the claims up to and including the time $t>0$ is therefore the sum
\[
S(t) =\sum_{i=1}^{N(t)}{X_i},
\]
where, by definition, $S(t)=0$ when $N(t)=0$. Hence, the claim-counting process $N(t)$, $t\ge 0$, is a right-continuous stochastic process with unit steps at the time instances $T_i$ and such that $N(0)=0$. Denote the inter-claim times by $V_i=T_i-T_{i-1}$, with $T_0=0$ by definition.

One of the most basic examples of such a claim-counting process is the homogenous Poisson process under which the inter-claim times are independent and exponentially distributed random variables. Assuming that the claim sizes are independent and identically distributed (i.i.d.), and independent of the claim-counting process $N(t)$, the aggregate claim size $S(t)$, $t\ge 0$, is a compound Poisson process (e.g., Klugman et al., 2008, Section 11.1.2). In this case we have the formulas
\begin{equation}
{\mathbf{E}[S(t)] \over t } =\mathbf{E}[X_1]\lambda_0
\label{mean-var-CPa}
\end{equation}
and
\begin{equation}
{\mathbf{Var}[S(t)] \over  t } =\mathbf{E}[X_1^2]\lambda_0,
\label{mean-var-CPb}
\end{equation}
where $\lambda_0 $ is the Poisson parameter.

More generally, assuming the Sparre Andersen model, for which the inter-claim times are i.i.d. random variables with a common arbitrary distribution and independent of the claim sizes, we have that
$\mathbf{E}[S(t)] = \mathbf{E}[N(t)]\mu $ and
$\mathbf{Var}[S(t)] = \mathbf{Var}[N(t)]\mu^2+\mathbf{E}[N(t)]\sigma^2$, where $\mu$ is the mean of $X_1$, and $\sigma^2$ is the variance of $X_1$.

In this paper we consider a very general claim-counting
process, which is called the order statistic point process or, for short, the OS-point process. It includes such classes of point processes as the non-homogenous Poisson process, the linear birth process, and a number of other ones
(see, e.g., Berg and Spizzichino, 2000, and Debrabant, 2008). We concentrate our attention in this direction of generality and thus deliberately restrict ourselves to only calculating the mean $\mathbf{E}[S(t)]$ and the second moment $\mathbf{E}[S^2(t)]$ of the aggregate claim $S(t)$, thus keeping mathematical complexities within reasonable limits and facilitating a clearer and more transparent exposition of main ideas, results and their proofs. If desired, the herein presented ideas and techniques can be extended to more complex functionals of $S(t)$ such as its Laplace transform, moment generating function, etc.

Furthermore, we allow each claim size $X_i$ depend on its arrival time $T_i$ as well as on the preceding inter-claim time $V_{i}$, or equivalently on $V_i$ and $T_{i-1}$, or perhaps just on the inter-claim time $V_{i}$. Such dependence structures reflect our belief that, in some situations, the longer a claim does not arrive, the larger the claim is expected to be. This is a common situation with claims resulting from earthquakes, hurricanes, and other catastrophic events.

The rest of the paper is organized as follows. In Section \ref{section-2} we recall basic facts about the order statistic point process, introduce some notation, and  derive general formulas for the mean $\mathbf{E}[S(t)]$ and the second moment $\mathbf{E}[S^2(t)]$. In Section \ref{section-3} we specialize these general formulas to the mixed Poisson process, and then even further specialize them by assuming the dependence structure between the claim sizes and inter-claim times as suggested by  Boudreault et al. (2006). These special cases, though obviously of independent interest, also illustrate how our general results of Section \ref{section-2} can be utilized in particular situations. Section \ref{section-4} develops analogous results for the second moment $\mathbf{E}[S^2(t)]$ and thus for the variance $\mathbf{Var}[S(t)]$ in the case of the mixed Poisson process and the dependence structure of Boudreault et al. (2006). Section \ref{section-5} concludes the paper. To facilitate an easier reading of main results, we have relegated most of the proofs to an appendix.

\section{The OS-point process and general formulas}
\label{section-2}

Following Crump (1975), we say that $N(t)$, $t\ge 0$, is an order statistic point process (i.e., OS-point process) if for every $t>0$ and $n\ge 0$ with the positive probability
\begin{equation}\label{eq: pi t,n}
\pi_{t,n}=\mathbf{P}[N(t)=n],
\end{equation}
we have that, conditionally on the event $N(t)=n$, the claim arrival times $T_1,\dots , T_n $ have the same joint distribution as the order statistics of i.i.d. random variables $\tau_1,\dots , \tau_n $ with a common cumulative distribution function (cdf) $F_t(x)$, $x \ge 0$, such that $F_t(t)=1$.

Crump (1975) has proved that the cdf $F_t(x)$ is related to the process $N(t)$, $t\ge 0$, via the equation
\begin{equation}
F_t(x)={\mathbf{E}[N(x)] \over \mathbf{E}[N(t)]},\quad 0\le x \le t.
\label{cdf-main}
\end{equation}
For example, the non-homogenous Poisson process $N(t)$, $t\ge 0$, with the finite cumulative intensity function $\Lambda(x)=\int_0^x \lambda(y)dy$, $x\ge 0$,  is an OS-point process, where $\lambda(y)$ is the intensity function. In this case we have that
\begin{equation}
F_t(x)={ \Lambda(x) \over \Lambda(t) },\quad 0\le x \le t.
\label{cdf-nhpp}
\end{equation}

Hence, the homogenous Poisson process with a constant intensity function $\lambda(y)\equiv \lambda_0>0$ is an OS-point process with the cdf
\begin{equation}
F_t(x)={x \over t},\quad 0\le x \le t.
\label{cdf}
\end{equation}
When conditioned on $N(t)=n$, the $n$ claim arrival times $T_i$ can be viewed as the order statistics of i.i.d.\, random variables $\tau_1,\dots, \tau_n$ with the common $[0,t]$-uniform distribution. When conditioned on $N(t)=n$, the $n$ inter-claim times $V_i=T_{i}-T_{i-1}$ can be viewed as $[0,t]$-uniform spacings.

The above example gives rise to the class of point processes with so-called property $\mathcal{P}$, which coincides with the class of mixed Poisson processes. Later in this paper we shall use the latter processes to illustrate general formulas derived in earlier sections. For further examples of OS-point processes, their generalizations, and additional references, we refer to Berg and Spizzichino (2000), and Debrabant (2008).

We note that actuaries have found mixed processes particularly useful for reasons such as modelling claims that fall into several categories (see, e.g., Klugman et al., 2008, Section 6.10). For instance, if we want to differentiate between car insurance policies depending on the primary driver's gender, age or driving history, it is appropriate to model data in each of these categories using different distributions and then mix them in order to provide a model for the entire pool of insurance policies. Due to its attractive properties, the mixed Poisson process is a particularly popular choice when modelling aggregate claim amounts (see, e.g., Rolski et al., 1999, Section 4.3.3).

Coming back to our general framework, the following theorem lays down foundations for calculating the expected aggregate loss $\mathbf{E}[S(t)]$ under various dependence structures between the claim sizes $X_i$, their arrival times $T_i$, and the previous arrival times $T_{i-1}$.

\begin{theorem}\label{theorem-0}
Assume that $N(t)$, $t\ge 0$, is the OS-point process with an absolutely continuous cdf $F_t$ and density $f_t$. Let the distribution of each claim size $X_i$ depend on the claim arrival process only via the times $T_{i-1}$ and $T_{i}$. Then
\begin{align}
\mathbf{E}[S(t)]
&=\sum_{n=1}^{\infty } \pi_{t,n} \Bigg \{ n \int_{0}^{t}\mathbf{E}\big [ X_1 | T_{1}=y \big ]  f_t(y)(1-F_t(y))^{n-1}dy
\notag
\\
&\hspace*{30mm} + \sum_{i=2}^n {n! \over (i-2)!(n-i)!}\int_{0}^{t}\int_{0}^{y}
\mathbf{E}\big [  X_i | T_{i-1}=x, T_{i}=y \big ]
\notag
\\
&\hspace*{40mm} \times f_t(x)f_t(y)F_t^{i-2}(x)(1-F_t(y))^{n-i}dxd y\Bigg \}.
\label{ii-4}
\end{align}
\end{theorem}

The proofs of this and subsequent results are relegated to an appendix due to their technical nature.

We next derive a general formula for the second moment $\mathbf{E}[S^2(t)]$, which together with the above formula for the mean $\mathbf{E}[S(t)]$ produces general formulas for the variance and the standard deviation of $S(t)$.

\begin{theorem}\label{theorem-02}
Let the conditions of Theorem \ref{theorem-0} be satisfied, and let the distribution of each product $X_iX_j$ depend only on the times $T_{i-1}$, $T_{i}$, $T_{j-1}$, $T_{j}$. Then
\begin{equation}
\mathbf{E}[S^2(t)]=\sum_{n=1}^{\infty } \pi_{t,n} \big ( A_{t,n}+2B_{t,n} \big ) ,
\label{ii-42}
\end{equation}
where
\begin{align*}
A_{t,n}&=n \int_{0}^{t}\mathbf{E}\big [ X_1^2 | T_{1}=y \big ]  f_t(y)(1-F_t(y))^{n-1}dy
\notag
\\
& \hspace*{2.5cm} + \sum_{i=2}^n {n! \over (i-2)!(n-i)!}\int_{0}^{t}\int_{0}^{y}
\mathbf{E}\big [  X_i^2 | T_{i-1}=x, T_{i}=y \big ]
\notag
\\
& \hspace*{7.7cm} \times f_t(x)f_t(y)F_t^{i-2}(x)(1-F_t(y))^{n-i}dxd y
\end{align*}
and
\begin{align*}
B_{t,n}&= n(n-1) \int_{0}^{t}\int_{0}^{z} \mathbf{E}\big [ X_1 X_2 | T_{1}=y ,T_{2}=z \big ]
f_t(y)f_t(z)(1-F_t(z))^{n-2}dy dz
\notag
\\
& \quad + \sum_{j=3}^{n} {n! \over (j-3)!(n-j)!} \int_{0}^{t}\int_{0}^{z}\int_{0}^{w}\mathbf{E}\big [ X_1 X_j |T_{1}=y ,T_{j-1}=w ,T_{j}=z  \big ]
\notag
\\
& \hspace*{4.5cm} \times f_t(y)f_t(w)f_t(z)(F_t(w)-F_t(y))^{j-3}(1-F_t(z))^{n-j}dydwdz
\notag
\\
&  \quad + \sum_{i=2}^{n-1}{n! \over (i-2)!(n-i-1)!} \int_{0}^{t}\int_{0}^{z}\int_{0}^{w}\mathbf{E}\big [ X_i X_{i+1} | T_{i-1}=y ,T_{i}=w ,T_{i+1}=z \big ]
\notag
\\
& \hspace*{6cm} \times f_t(y)f_t(w)f_t(z)F_t(y)^{i-2}(1-F_t(z))^{n-i-1}dydwdz
\notag
\\
&  \quad + \sum_{i=2}^{n-2}\sum_{j=i+2}^{n}{n! \over (i-2)!(j-i-2)!(n-j)!}
\notag
\\
&  \hspace*{1cm} \times \int_{0}^{z}\int_{0}^{w}\int_{0}^{y}
\mathbf{E}\big [ X_i X_j | T_{i-1}=x, T_{i}=y ,T_{j-1}=w , T_{j}=z \big ]
\notag
\\
&  \hspace*{1.5cm} \times f_t(x)f_t(y)f_t(w)f_t(z) F_t^{i-2}(x)(F_t(w)-F_t(y))^{j-i-2}(1-F_t(z))^{n-j} dxdydwdz .
\end{align*}
\end{theorem}

Theorems \ref{theorem-0} and \ref{theorem-02} allow us to calculate the first and second moments, and thus in turn the variance and the standard deviation, of the aggregate loss $S(t)$ based on the conditional expectations $\mathbf{E}[ X_i | T_{i-1}, T_{i}]$ and  $\mathbf{E}[ X_i X_j | T_{i-1}, T_{i} ,T_{j-1} , T_{j}]$ as well as on the density $f_t$ and the cdf $F_t$. We shall provide specific calculations of such quantities in the following sections, where we assume that the claim arrival process is the OS-point process with property  $\mathcal{P}$.

\section{Property $\mathcal{P}$ and the mean aggregate loss}
\label{section-3}

Following Feigin (1979), we say that a counting point process  $N(t)$, $t\ge 0$, satisfies the property $\mathcal{P}$ if it is an OS-point process and the cdf $F_t$ is given by formula \eqref{cdf}. Hence, the homogenous Poisson process is a counting process with property $\mathcal{P}$.

In general, Feigin (1979) has proved that a counting point process has property $\mathcal{P}$ if and only if it is a mixed Poisson process. With $L(\lambda )$, $\lambda\ge 0$, denoting the structure cdf of this process, the probabilities $\pi_{t,n}$ -- which we from now on denote by $\pi_{t,n}^L$ -- take on the form
\begin{equation}
\pi_{t,n}^{L}= \int_{0}^{\infty}{(\lambda t)^n \over n!} e^{-\lambda t}dL(\lambda).
\label{mixed-p}
\end{equation}
For example, when the structure cdf is degenerate at a point $\lambda_0>0$, that is, $L(\lambda )= \mathbf{1}\{\lambda_0 \le \lambda \}$, then $\pi_{t,n}^L$, $n\ge 0$, are the usual Poisson probabilities with the parameter $\lambda_0$.

The next theorem is a variant of Theorem \ref{theorem-0} in the case of the mixed Poisson process. To facilitate a straightforward application of the theorem for analyzing models such as the one when each claim size $X_i$ depends only on the preceding inter-claim time $V_i$ (Theorem \ref{theorem-2} below), we formulate our next theorem in terms of $X_i$ dependent on the pair $(T_{i-1},V_i)$. Of course, both pairs $(T_{i-1},V_i)$ and $(T_{i-1},T_i)$ generate same sigma-algebras and thus convey same information about the claim size $X_i$.

\begin{theorem}\label{theorem-1}
Assume that $N(t)$, $t\ge 0$, is the mixed Poisson process with a structure cdf $L(\lambda)$, $\lambda\ge 0$. Let each claim size $X_i$ depend only on the time $T_{i-1}$ of the previous claim and the inter-claim time $V_i=T_i-T_{i-1}$. Then
\begin{multline}
\mathbf{E}[S(t)]
=\sum_{n=1}^{\infty } \pi_{t,n}^{L} \Bigg \{ n \int_{0}^{t}\mathbf{E}\big [ X_1 | T_{1}=y \big ] {(t-y)^{n-1}\over t^n} dy
\\
 + \sum_{i=2}^n {n! \over (i-2)!(n-i)!}\int_{0}^{t} \int_{0}^{t-x}
\mathbf{E}\big [ X_i | T_{i-1}=x, V_{i}=v \big ]
 {x^{i-2}(t-x-v)^{n-i}\over t^n} dv dx\Bigg \} .
\label{ii-6}
\end{multline}
\end{theorem}

The next theorem is a special case of Theorem \ref{theorem-1} when each claim size $X_i$ depends only on the preceding inter-claim time $V_i$. This has been a popular model in the actuarial literature (e.g., Boudreault et al., 2006, and references therein). In what follows, we use the notation
\[
Q(n|\,v)=\sum_{i=1}^{n} \Delta_i(v) \quad \textrm{with} \quad \Delta_i(v)=\mathbf{E} [ X_i | V_i=v  ].
\]

\begin{theorem}\label{theorem-2}
Assume that $N(t)$, $t\ge 0$, is the mixed Poisson process with a structure cdf $L(\lambda)$, $\lambda\ge 0$. Let each claim size $X_i$ depend only on the inter-claim time $V_i$. Then
\begin{equation}
\mathbf{E}[S(t)]
=\int_{0}^{\infty}  \int_{0}^t \lambda e^{-\lambda v}
\mathbf{E}\big [Q\big (N_{\lambda }(t-v)+1|\,v \big )\big ]\, dv\,dL(\lambda),
\label{th-0aa}
\end{equation}
where $N_{\lambda }$ is the homogenous Poisson process with the constant rate $\lambda >0 $.
\end{theorem}

The following corollary is a simplification of Theorem \ref{theorem-2} when
the conditional random variables $X_i|V_i=v$ have same means. The corollary is  slightly more general than equation (3) of Boudreault et al. (2006), which we shall formulate in a moment.

\begin{corollary}\label{cor-0}
Let the assumptions of Theorem \ref{theorem-2} be satisfied, and let the conditional random variables $X_i|V_i=v$ have same means. Then
\begin{equation}
\mathbf{E}[S(t)]
=\int_{0}^{\infty}  \int_{0}^t \lambda e^{-\lambda v}\mathbf{E}[X_1|V_1=v]
\big ( \lambda(t-v)+1 \big )\, dv\,dL(\lambda).
\label{th-0bb}
\end{equation}
\end{corollary}

\begin{corollary}[\rm Boudreault et al., 2006, eq.~(3)]\label{cor-1}
Assume that $N(t)$, $t\ge 0$, is a homogenous Poisson process with a constant rate $\lambda>0$. Let each claim size $X_i$ depend only on the inter-claim time $V_i$, and let the conditional random variables $X_i|V_i=v$ have same means. Then
\begin{equation}
\mathbf{E}[S(t)]
=\int_{0}^t \lambda e^{-\lambda v}\mathbf{E}[X_1|V_1=v] \big ( \lambda(t-v)+1 \big ) dv.
\label{eq-12}
\end{equation}
Alternatively, integrating the right-hand side of equation \eqref{eq-12} by parts, we have
\begin{equation}\label{mas}
\mathbf{E}[S(t)] = \lambda t \bigg ( \mathbf{E}[X_1|V_1=0]+\int_{0}^{t}{e^{-\lambda v}\left(1-\frac{v}{t} \right)d\mathbf{E}[X_1|V_1=v]} \bigg ),
\end{equation}
which appears as identity (3) in Boudreault et al. (2006).
\end{corollary}

Boudreault et al. (2006) assume that each claim size $X_i$ depends only on the preceding inter-claim time $V_i$, and that the conditional random variables $X_i|V_i=v$ are identically distributed with their common cdf given by the formula
\begin{equation}\label{eqcd}
\mathbf{P}[X_i\le y|V_i=v]=\left(1-e^{-\beta v} \right)H_{\ell}(y)+ e^{-\beta v}H_s(y),
\end{equation}
where $H_{\ell}$ and $H_s$ are two cdf's, and $\beta \ge 0$ is a parameter. Since larger values of $V_i$ result in smaller $e^{-\beta V_i}$, it is natural to view $H_{\ell}$ as the cdf of larger claims and $H_s$ as the cdf of smaller claims.
Under this dependence model, we have the following corollary with $Y_{\ell}$ and $Y_s$ denoting random variables with the cdf's $H_{\ell}$ and $H_s$, respectively.

\begin{corollary} \label{cor-4}
Assume that $N(t)$, $t\ge 0$, is the mixed Poisson process with a structure cdf $L(\lambda)$, $\lambda\ge 0$. Let each claim size $X_i$ depend only on the inter-claim time $V_i$, and let the distribution of the conditional random variables $X_i|V_i$ be given by formula \eqref{eqcd}. Then
\begin{equation}\label{formula-1}
\mathbf{E}[S(t)] = \mathbf{E}[Y_{\ell}]\int_{0}^{\infty} \big(
\lambda t-\mathfrak{A}(t,\lambda,\beta ) \big) \,dL(\lambda)
+
\mathbf{E}[Y_s]\int_{0}^{\infty} \mathfrak{A}(t,\lambda,\beta ) \,dL(\lambda),
\end{equation}
where
\[
\mathfrak{A}(t,\lambda,\beta )=\frac{\lambda }{(\beta + \lambda )^2}\,\left( \beta  -
      \frac{\beta }{e^{(\beta + \lambda ) t}} +
      (\beta + \lambda) \lambda t \right).
\]
\end{corollary}

Under the conditions of Corollary \ref{cor-4}, the following four statements hold:
\begin{enumerate}
\item
We have that, when $t\to \infty $,
\begin{equation}\label{lim-a}
{\mathfrak{A}(t,\lambda,\beta )\over t }= \frac{\lambda^2 }{\beta  + \lambda } +o(1) .
\end{equation}
A closer look at the remainder term $o(1)$ shows that we can apply Lebesgue's dominated convergence theorem and obtain the limit
\begin{equation}\label{formula-1ba11}
\lim_{t\to \infty }{\mathbf{E}[S(t)] \over t }
=\mathbf{E}[Y_{\ell}] \int_{0}^{\infty}\left(\frac{\beta}{\beta  + \lambda }
 \right)\lambda\,dL(\lambda)
+
\mathbf{E}[Y_s]\int_{0}^{\infty}\left(\frac{\lambda }{\beta  + \lambda } \right)\lambda\,dL(\lambda) .
\end{equation}
\item
If there is only one type of claims, that is, $Y_\ell$ and $Y_s$ have same distributions as a random variable $Y$, then for every $t\ge 0$ and regardless of the value of $\beta $, we have that
\begin{equation}
{\mathbf{E}[S(t)] \over t }
=\mathbf{E}[Y]\int_{0}^{\infty} \lambda \,dL(\lambda).
\label{formula-1ba22}
\end{equation}
Note that the right-hand side of equation (\ref{formula-1ba22}) can be written as $ \mathbf{E}[Y]\mathbf{E}[\Lambda ]$, where $\Lambda $ is a random variable with the cdf $L$.
\item
 If the structure cdf is degenerate at the point $\lambda_0$, that is,  $L(\lambda )=\mathbf{1}\{\lambda_0 \le \lambda \}$, then from equation (\ref{formula-1ba11}) we have that
\begin{equation}\label{formula-1ba33}
\lim_{t\to \infty }{\mathbf{E}[S(t)] \over  t }
= \mathbf{E}[Y_{\ell}]\left(\frac{\beta }{\beta  + \lambda_0 }
 \right)\lambda_0
+
\mathbf{E}[Y_s]\left(\frac{\lambda_0 }{\beta  + \lambda_0 }\right)\lambda_0.
\end{equation}
\item
 If the structure cdf is degenerate at the point $\lambda_0$ and there is only one type of claims, then from equation (\ref{formula-1ba22}) we have that, for every $t\ge 0$,
\begin{equation}\label{formula-1ba44}
{\mathbf{E}[S(t)] \over  t }
= \mathbf{E}[Y]\lambda_0 .
\end{equation}
We can view this result as a `mixed-Poisson analogue' of equation (\ref{mean-var-CPa}).
\end{enumerate}

\section{Property $\mathcal{P}$ and the variability of aggregate losses}
\label{section-4}

In this section we investigate the variability of the aggregate loss $S(t)$ in the case of the counting process with property $\mathcal{P}$, that is, when the process  $N(t)$, $t\ge 0$, is mixed Poisson. We choose the variance of $S(t)$ to measure the variability, and for this, given the results of the previous section, we only need to derive formulas for the second moment $\mathbf{E}[S^2(t)]$.

We start with a reformulation of Theorem \ref{theorem-02} under the assumption that the counting process is mixed Poisson and claim sizes $X_i$ are governed by $T_{i-1}$ and $V_i=T_i-T_{i-1}$, instead of $T_{i-1}$ and $T_{i}$.

\begin{theorem}\label{th-new-12}
Assume that $N(t)$, $t\ge 0$, is the mixed Poisson process with a structure cdf $L(\lambda)$, $\lambda\ge 0$. Let each claim size $X_i$ depend on the claim arrival process only via the time $T_{i-1}$ of the previous claim and the inter-claim time $V_i=T_i-T_{i-1}$, and let the distribution of the product $X_iX_j$ depend on the claim times $T_{i-1}$ and $T_{j-1}$ as well as on the inter-claim times $V_{i}$ and $V_{j}$. Then
\begin{equation}
\mathbf{E}[S^2(t)]=\sum_{n=1}^{\infty }\pi_{t,n}^{L} \big ( A_{t,n}^{L}+2B_{t,n}^{L} \big ) ,
\label{ii-42a}
\end{equation}
where
\begin{multline*}
A_{t,n}^{L}= n \int_{0}^{t}\mathbf{E}\big [ X_1^2 | T_{1}=y \big ] {(t-y)^{n-1}\over t^n} dy
\\
 + \sum_{i=2}^n {n! \over (i-2)!(n-i)!}\int_{0}^{t} \int_{0}^{t-x}
\mathbf{E}\big [ X_i^2 | T_{i-1}=x, V_{i}=v \big ]
 {x^{i-2}(t-x-v)^{n-i}\over t^n} dv dx
\end{multline*}
and
\begin{align}
B_{t,n}^{L}&= n(n-1) \int_{0}^{t} \int_{0}^{t-y} \mathbf{E}\big [ X_1 X_2 | T_{1}=y ,V_{2}=v \big ]
{(t-y-v)^{n-2} \over t^n}dv dy
\notag
\\
&\qquad + \sum_{j=3}^{n} {n! \over (j-3)!(n-j)!} \int_{0}^{t}\int_{y}^{t} \int_{0}^{t-w} \mathbf{E}\big [ X_1 X_j |T_{1}=y ,T_{j-1}=w ,V_{j}=v  \big ]
\notag
\\
& \hspace*{45mm} \times {(w-y)^{j-3}(t-w-v)^{n-j}\over t^n}dvdwdy
\notag
\\
&\qquad + \sum_{i=2}^{n-1}{n! \over (i-2)!(n-i-1)!} \int_{0}^{t}\int_{0}^{t-y} \int_{0}^{t-y-u}\mathbf{E}\big [ X_i X_{i+1} | T_{i-1}=y ,V_{i}=u ,V_{i+1}=v \big ]
\notag
\\
& \hspace*{45mm} \times {y^{i-2}(t-y-u-v)^{n-i-1}\over t^n}dvdudy
\notag
\\
&\qquad + \sum_{i=2}^{n-2}\sum_{j=i+2}^{n}{n! \over (i-2)!(j-i-2)!(n-j)!}
\notag
\\
& \qquad \times \int_{0}^{t}\int_{0}^{t-x} \int_{x+u}^{t}\int_{0}^{t-w}
\mathbf{E}\big [ X_i X_j | T_{i-1}=x, V_{i}=u ,T_{j-1}=w , V_{j}=v \big ]
\notag
\\
& \hspace*{45mm} \times { x^{i-2}(w-x-u)^{j-i-2}(t-w-v)^{n-j} \over t^n}dvdwdudx.
\end{align}
\end{theorem}

Theorem \ref{th-new-12} is a general result that allows us to tackle various dependence structures between $X_i$ and the pair $(T_{i-1},V_i)$. In particular, in the next theorem we present a formula for the second moment $\mathbf{E}[S^2(t)]$ assuming that each $X_i$ depends only on $V_i$. Naturally, in this special case we get a shorter and more elegant formula than the one in Theorem \ref{th-new-12}. We use the notation
\[
\Theta(n|v)=\sum_{i=1}^n \Theta_i(v) \quad \textrm{with} \quad \Theta_i(v)=\mathbf{E} [ X_i^2 | V_i=v  ]
\]
and
\[
\Upsilon(n|y,v)=\sum_{1\le i \ne j \le n}\Delta_i(y)\Delta_j(v) \quad \textrm{with} \quad \Delta_i(v)=\mathbf{E} [ X_i | V_i=v  ].
\]

\begin{theorem}\label{theorem-12bb}
Assume that $N(t)$, $t\ge 0$, is the mixed Poisson process with a structure cdf $L(\lambda)$, $\lambda\ge 0$. Let each claim size $X_i$ depend only on the inter-claim time $V_i$, and let the conditional variables $X_i|V_i$ be independent. Then
\begin{multline}
\mathbf{E}[S^2(t)]
=\int_{0}^{\infty}  \int_{0}^t \lambda e^{-\lambda v}
\mathbf{E}\big [\Theta \big (N_{\lambda }(t-v)+1|\,v \big )\big ]\, dv\,dL(\lambda)
\\
+ \int_{0}^{\infty}\int_{0}^{t} \int_{0}^{t-y}\lambda^2 e^{-\lambda (y+v)} \mathbf{E}\big [\Upsilon(N_{\lambda }(t-v)+2|\,y,v)\big ]\,dv dy dL(\lambda) .
\label{ii-42ddd}
\end{multline}
\end{theorem}

We next specialize Theorem \ref{theorem-12bb} even further.

\begin{corollary}\label{cor-12bb}
Assume that $N(t)$, $t\ge 0$, is the mixed Poisson process with a structure cdf $L(\lambda)$, $\lambda\ge 0$. Let each claim size $X_i$ depend only on the inter-claim time $V_i$, and let the conditional variables $X_i|V_i=v$ be independent and have same first and second moments. Then we have that
\begin{multline}
\mathbf{E}[S^2(t)]
=\int_{0}^{\infty}  \int_{0}^t \lambda e^{-\lambda v} \Theta_1(v)
\big (\lambda (t-v)+1\big ) dv\,dL(\lambda)
\\
+ \int_{0}^{\infty}\int_{0}^{t} \int_{0}^{t-y}\lambda^2 e^{-\lambda (y+v)}
\Delta_1(y)\Delta_1(v)\left(\big (\lambda (t-y-v)+2 \big )^2-2\right)
dv dy dL(\lambda).
\label{ii-42ddd2}
\end{multline}
\end{corollary}

Corollary \ref{cor-12bb} connects our earlier results concerning $\mathbf{E}[S^2(t)]$ with the dependence model specified by formula \eqref{eqcd}. This makes the contents of the next corollary.

\begin{corollary} \label{cor-4g}
Assume that $N(t)$, $t\ge 0$, is the mixed Poisson process with a structure cdf $L(\lambda)$, $\lambda\ge 0$. Let each claim size $X_i$ depend only on the inter-claim time $V_i$, and let the distribution of the conditional variables $X_i|V_i$ be given by formula \eqref{eqcd}. Then with $Y_{\ell}$ and $Y_s$ denoting random variables with cdfs $H_{\ell}$ and $H_s$, respectively, we have that
\begin{align}
\mathbf{E}[S^2(t)]
&= \mathbf{E}[Y_{\ell}^2]\int_{0}^{\infty} \big(
\lambda t-\mathfrak{A}(t,\lambda,\beta ) \big) \,dL(\lambda)
 +
\mathbf{E}[Y_s^2]\int_{0}^{\infty} \mathfrak{A}(t,\lambda,\beta ) \,dL(\lambda)
\notag
\\
& \qquad  + \big (\mathbf{E}[Y_{\ell}]\big )^2 \int_{0}^{\infty}
\Big ( \mathfrak{B}(t,\lambda,0,0 )-2\mathfrak{B}(t,\lambda,0,\beta )+\mathfrak{B}(t,\lambda,\beta,\beta ) \Big ) dL(\lambda)
\notag
\\
&  \qquad + 2 \mathbf{E}[Y_{\ell}]\mathbf{E}[Y_s]\int_{0}^{\infty} \Big ( \mathfrak{B}(t,\lambda,0,\beta )-\mathfrak{B}(t,\lambda,\beta,\beta ) \Big ) dL(\lambda)
\notag
\\
&  \qquad + \big ( \mathbf{E}[Y_s]\big )^2 \int_{0}^{\infty} \mathfrak{B}(t,\lambda,\beta,\beta )  dL(\lambda),
\label{ii-42ddd4}
\end{align}
where $\mathfrak{A}(t,\lambda,\beta )$ is defined in Corollary \ref{cor-4}, and the $\mathfrak{B}$-quantities are as follows:
\begin{equation}
\mathfrak{B}(t,\lambda,\beta,\beta )
={ 2\beta \lambda^2(\beta-2\lambda ) \over (\beta + \lambda )^4}
+t { 4\beta \lambda^3 \over (\beta + \lambda )^3}
+t^2 { \lambda^4 \over (\beta + \lambda )^2}
-\frac{2\,\beta {\lambda }^2\left( -2\,\lambda  +
           \beta \,\left( 1 +
              t\,(\beta + \lambda )  \right)  \right)
         }{(\beta + \lambda )^4 e^{t\,(\beta + \lambda ) }}
\label{f-1}
\end{equation}
and
\begin{equation}
\mathfrak{B}(t,\lambda,0,\beta )
= -\frac{ 2 \beta\lambda ^2 }{(\beta + \lambda )^3}
+ t\frac{ 2\,\beta {\lambda }^2}{(\beta + \lambda )^2}
+ t^2\frac{ \lambda ^3  }{\beta + \lambda }
+\frac{ 2 \beta\lambda ^2 }{(\beta + \lambda )^3e^{t\,(\beta +\lambda )}} ,
\label{f-3}
\end{equation}
with both formulas (\ref{f-1}) and (\ref{f-3}) in the case $\beta = 0 $ reducing to
\begin{equation}
\mathfrak{B}(t,\lambda,0,0 )=  t^2\,{\lambda }^2 .
\label{f-2}
\end{equation}
\end{corollary}

We next analyze the result of Corollary \ref{cor-4g} in detail. First, we recall the asymptotic formula (\ref{lim-a}). Then we check that, when $t\to \infty $,
\begin{equation}
{\mathfrak{B}(t,\lambda,\beta,\beta )\over t}
=
\frac{4\,\beta \lambda ^3  }{(\beta + \lambda )^3}
    + t \,\frac{\lambda^4 }{(\beta + \lambda )^2} +o(1)
\label{f-1ii}
\end{equation}
and
\begin{equation}
{\mathfrak{B}(t,\lambda,0,\beta )\over t}
=
\frac{2\,\beta \lambda ^2  }{(\beta + \lambda )^2}
    + t \,\frac{\lambda^3 }{\beta + \lambda}+o(1).
\label{f-3ii}
\end{equation}
Of course, equation (\ref{f-2}) gives
\begin{equation}
{\mathfrak{B}(t,\lambda,0,0 ) \over t}=  t\,{\lambda }^2 .
\label{f-2ii}
\end{equation}
A closer look at the two remainder terms $o(1)$ in equations (\ref{f-1ii}) and (\ref{f-3ii}) shows that we can apply Lebesgue's dominated convergence theorem and then,  using Corollary \ref{cor-4g}, we obtain the limit
\begin{align}
\lim_{t\to \infty }{\mathbf{E}[S^2(t)] \over t }
&= \mathbf{E}[Y_{\ell}^2] \int_{0}^{\infty}\left(\frac{\beta}{\beta  + \lambda }
 \right)\lambda\,dL(\lambda)
+
\mathbf{E}[Y_s^2]\int_{0}^{\infty}\left(\frac{\lambda }{\beta  + \lambda } \right)\lambda\,dL(\lambda)
\notag
\\
& \qquad  + \big (\mathbf{E}[Y_{\ell}]\big )^2 \int_{0}^{\infty}
\bigg ( -\frac{4\,\beta^2 \lambda ^2  }{(\beta + \lambda )^3}
    +t  \frac{\beta^2 \lambda ^2  }{(\beta + \lambda )^2}
    \bigg ) dL(\lambda)
\notag
\\
&  \qquad + \mathbf{E}[Y_{\ell}]\mathbf{E}[Y_s]\int_{0}^{\infty}
\bigg ( \frac{4\,\beta \lambda ^2(\beta- \lambda )}{(\beta + \lambda )^3}
    +t  \frac{2 \beta \lambda ^3  }{(\beta + \lambda )^2}
    \bigg )  dL(\lambda)
\notag
\\
&  \qquad + \big ( \mathbf{E}[Y_s]\big )^2 \int_{0}^{\infty}
\bigg ( \frac{4\,\beta \lambda ^3  }{(\beta + \lambda )^3}
    + t \,\frac{\lambda^4 }{(\beta + \lambda )^2}
     \bigg ) dL(\lambda).
\label{new-1}
\end{align}

It follows that, under the conditions of Corollary \ref{cor-4g}, the following four statements hold:
\begin{enumerate}
\item
We have that
\begin{equation}
\lim_{t\to \infty }{\mathbf{E}[S^2(t)] \over t^2 }
=  \int_{0}^{\infty}\bigg ( \mathbf{E}[Y_{\ell}]{\beta \over \beta + \lambda  }+\mathbf{E}[Y_s] { \lambda \over \beta + \lambda  }\bigg )^2\lambda^2 dL(\lambda).
\label{new-2}
\end{equation}
\item
If there is only one type of claims, that is, $Y_\ell$ and $Y_s$ are distributed as a random variable $Y$, then \
\begin{equation}
\lim_{t\to \infty }{\mathbf{E}[S^2(t)] \over t^2 }
=  (\mathbf{E}[Y])^2\int_{0}^{\infty} \lambda^2 dL(\lambda).
\label{new-3}
\end{equation}
Note that the right-hand side of equation (\ref{formula-1ba22}) can be rewritten as $ (\mathbf{E}[Y])^2\mathbf{E}[\Lambda^2 ]$, where $\Lambda $ is a random variable with the cdf $L(\lambda )$.
\item
Combining statements (\ref{formula-1ba22}) and (\ref{new-1}), we have that
\begin{equation}
\lim_{t\to \infty }{\mathbf{Var}[S(t)] \over t^2 }
=  (\mathbf{E}[Y])^2\mathbf{Var}[\Lambda ].
\label{new-4}
\end{equation}
Note that when the structure cdf $L(\lambda )$ is degenerate at a point $\lambda_0>0$, then $\mathbf{Var}[\Lambda ]=0$ and so statement (\ref{new-4}) implies that $\mathbf{Var}[S(t)] /t^2\to 0$ when $t\to \infty $.
\item
Statements (\ref{formula-1ba22}) and (\ref{new-1}) help us to establish an asymptotic formula for the variance $\mathbf{Var}[S(t)]$ in the case $\mathbf{Var}[\Lambda ]=0$. Namely, we have that
\begin{equation}
\lim_{t\to \infty }{\mathbf{Var}[S(t)]\over t }=(\mathbf{E}[Y])^2\mathbf{E}[\Lambda ].
\label{new-5}
\end{equation}
We can view this result as a `mixed-Poisson analogue' of equation  (\ref{mean-var-CPb}).
\end{enumerate}

\section{Conclusions}
\label{section-5}

We have demonstrated that the order statistic point process -- which is a remarkably general process -- provides a tractable model for insurance claim arrivals and enables us to calculate various quantities of interest. More complex quantities than the herein tackled first and second moments of aggregate claims can be studied using  techniques of the present paper, which can further be extended and generalized to incorporate even more general dependence structures between claim sizes and their (inter-)arrival times. In summary, we believe that the herein suggested methodology opens up a fruitful direction for research in Ruin Theory and beyond.



\section*{References}
\def\hang{\hangindent=\parindent\noindent}

\hang
Abramowitz, M. and Stegun, I.A., (1972).
\textit{Handbook of Mathematical Functions with Formulas, Graphs, and Mathematical Tables.} Dover, New York.

\hang
Berg, M. and Spizzichino, F.  (2000).
Time-lagged point processes with the order-statistics property.
\textit{Mathematical Methods of Operations Research} 51, 301--314.

\hang
Boudreault, M., Cossette, H., Landriault, D., and Marceau, E. (2006). On a risk model with dependence between inter-claim arrivals and claim sizes.
\textit{Scandinavian Actuarial Journal} 5, 256-285.

\hang
Crump, K.S. (1975).
On point processes having an order statistic structure.
\textit{Sankhy\={a}, Series A}, 37, 396--404.

\hang
David, H.A. and Nagaraja, H.N. (2003). \textit{Order Statistics}. (Third edition.) Wiley, New Jersey.

\hang
Debrabant, B. (2008). \textit{Point Processes with a Generalized Order Statistic Property.} Logos, Berlin.

\hang
Feigin, P.D.  (1979).
On the characterization of point processes with the order statistic property.
\textit{Journal of Applied Probability} 16, 297--304.

\hang
Garrido, J.  and Lu, Y. (2004). On double periodic non-homogeneous Poisson processes.
{\it Bulletin of the Association of Swiss Actuaries} 2, 195--212.

\hang
Huang, W.J. and Shoung, J.M.  (1994).
On a study of some properties of point processes. \textit{Sankhy\={a}, Series A}, 56, 67--76,

\hang
Klugman, S.A. Panjer, H.H., and Willmot, G.E. (2008).
\textit{Loss Models: From Data to Decisions}. (3rd edition.) Wiley, New Jersey.

\hang
L\'eveill\'e, G., Garrido, J., and Wang, Y.F. (2010).
Moment generating functions of compound renewal sums with discounted claims.
\textit{Scandinavian Actuarial Journal} 3, 165--184.

\hang
Li, S. (2008).
Discussion on ``On the Laplace transform of the aggregate discounted claims with Markovian arrivals.''
\textit{North American Actuarial Journal} 4, 443--445.

\hang
Lu, Y. and Garrido, J. (2005).
Doubly periodic non-homogeneous Poisson models for hurricane data.
\textit{Statistical Methodology} 2, 17--35.

\hang
Lu, Y. and Garrido, J. (2006).
Regime-switching periodic non-homogeneous Poisson processes.
\textit{North American Actuarial Journal} 10 (4), 235--248.

\hang
Ren, J. (2008).
On the Laplace transform of the aggregate discounted claims with Markovian arrivals.
\textit{North American Actuarial Journal} 2, 198--207.

\hang
Rolski, T., Schmidli, H., Schmidt, V., and Teugels, J. (1999).
\textit{Stochastic Processes for Insurance and Finance}. Wiley, Chichester.

\hang
Teugels, J.L. and Vynckier, P.  (1996).
The structure distribution in a mixed Poisson process. \textit{Journal of Applied Mathematics and Stochastic Analysis} 9, 489--496.

\appendix

\section{Appendix: Proofs}
\label{app}

\def\theequation{A.\arabic{equation}}

\begin{proof}[Proof of Theorem \ref{theorem-0}]
Using repeated conditioning, we obtain that
\begin{align}
\mathbf{E}[S(t)]
&= \sum_{n=1}^{\infty } \pi_{t,n}\bigg \{\sum_{i=1}^n \mathbf{E}\big [ X_i | N(t)=n \big ]\bigg \}
\notag
\\
&=\sum_{n=1}^{\infty } \pi_{t,n} \bigg \{\int_{0}^{t}\mathbf{E}\big [ X_1 | T_{1}=y , N(t)=n \big ]dF_{1|t,n}(y)
\notag
\\
& \qquad + \sum_{i=2}^n \iint_{0\le x\le y\le t }
\mathbf{E}\big [ X_i | T_{i-1}=x, T_{i}=y , N(t)=n \big ] dF_{i-1,i|t,n}(x, y)\bigg \},
\label{ii-1}
\end{align}
where
\begin{equation}
F_{i|t,n}(x)=\mathbf{P}[T_i\le v|N(t)=n],\quad i\ge 1,
\label{cdf-i}
\end{equation}
and
\begin{equation}
F_{i,j|t,n}(x,y)=\mathbf{P}[T_{i}\le x, T_{j}\le y|N(t)=n], \quad 1\le i\le j.
\label{cdf-ij}
\end{equation}
Conditionally on $N(t)=n$, the random variables $T_1,\dots , T_n$ can be viewed as the order statistics $\tau_{1:n}\le \dots \le \tau_{n:n}$ of i.i.d.\, random variables $\tau_1,\dots , \tau_n$ with the common cdf $F_t$ given by formula \eqref{cdf-main}. Hence, $F_{i|t,n}(x)=\mathbf{P}[\tau_{i:n}\le x]$ and
$F_{i,j|t,n}(x,y)=\mathbf{P}[\tau_{i:n}\le x, \tau_{j:n}\le y]$.
Consequently (see, e.g., David and Nagaraja, 2003), the density of $F_{1|t,n}$ is equal to
\begin{equation}
f_{1|t,n}(x)=nf_t(x)(1-F_t(x))^{n-1}
\label{ii-2}
\end{equation}
for all $x\in [0,t]$ and vanishes for all other $x$. For $2\le i \le n $, the density of $F_{i-1,i|t,n}$ is
\begin{equation}
f_{i-1,i|t,n}(x,y)={n! \over (i-2)!(n-i)!} f_t(x)f_t(y)F_t^{i-2}(x)(1-F_t(y))^{n-i}
\label{ii-3}
\end{equation}
for all $(x,y)$ such that $0\le x \le y \le t$ and vanishes for all other $(x,y)$. Plugging expressions (\ref{ii-2}) and (\ref{ii-3}) into equation \eqref{ii-1} gives (\ref{ii-4}) and concludes the proof of Theorem \ref{theorem-0}.
\end{proof}

\begin{proof}[Proof of Theorem \ref{theorem-02}]
Conditioning yields
\begin{equation}
\mathbf{E}[S^2(t)]
= \sum_{n=1}^{\infty } \pi_{t,n} \bigg \{ \sum_{i=1}^n \mathbf{E}\big [ X_i^2 | N(t)=n \big ]\bigg \} + 2 \sum_{n=1}^{\infty } \pi_{t,n} \bigg \{ \sum_{1\le i < j \le n}\mathbf{E}\big [ X_i X_j | N(t)=n \big ] \bigg \}.
\label{eq-10}
\end{equation}
We next calculate the two sums in the curly brackets on the right-hand side of equation (\ref{eq-10}). Starting with the first sum and proceeding just like in the proof of Theorem \ref{theorem-0} but now with $X_i^2$ instead of $X_i$, we obtain
\begin{align}
\sum_{i=1}^n \mathbf{E}\big [ X_i^2 | N(t)=n \big ]
&=\int_{0}^{t}\mathbf{E}\big [ X_1^2 | T_{1}=y , N(t)=n \big ]dF_{1|t,n}(y)
\notag
\\
& \qquad + \sum_{i=2}^n \int_{0}^{t}\int_{0}^{y}
\mathbf{E}\big [ X_i^2 | T_{i-1}=x, T_{i}=y , N(t)=n \big ] dF_{i-1,i|t,n}(x, y)
\notag
\\
&=n \int_{0}^{t}\mathbf{E}\big [ X_1^2 | T_{1}=y \big ]  f_t(y)(1-F_t(y))^{n-1}dy
\notag
\\
&  \qquad   + \sum_{i=2}^n {n! \over (i-2)!(n-i)!}\int_{0}^{t}\int_{0}^{y}
\mathbf{E}\big [  X_i^2 | T_{i-1}=x, T_{i}=y \big ]
\notag
\\
& \hspace*{3cm} \times f_t(x)f_t(y)F_t^{i-2}(x)(1-F_t(y))^{n-i}dxd y,
\label{eq-11}
\end{align}
where the right-most equation follows from equations \eqref{ii-2} and \eqref{ii-3}. As to the sum in the second curly brackets on the right-hand side of equation (\ref{eq-10}), we decompose it as follows
\begin{align}
\sum_{1\le i < j \le n}\mathbf{E}\big [ X_i X_j | N(t)=n \big ]
& =\mathbf{E}\big [ X_1 X_2 | N(t)=n \big ]
+\sum_{j=3}^{n}\mathbf{E}\big [ X_1 X_j | N(t)=n \big ]
\notag
\\
&\quad + \sum_{i=2}^{n-1}\mathbf{E}\big [ X_i X_{i+1} | N(t)=n \big ]
+ \sum_{i=2}^{n-2}\sum_{j=i+2}^{n}\mathbf{E}\big [ X_i X_j | N(t)=n \big ]
\label{eq-12v}
\end{align}
and then investigate the resulting four summands separately. We begin with the expectation
\begin{align}
\mathbf{E}\big [ X_1 X_2 | N(t)=n \big ]
&= \iint\mathbf{E}\big [ X_1 X_2 | T_{1}=y ,T_{2}=z \big ]dF_{1,2|t,n}(y,z)
\notag
\\
&= n(n-1) \int_{0}^{t}\int_{0}^{z} \mathbf{E}\big [ X_1 X_2 | T_{1}=y ,T_{2}=z \big ]
\notag
\\
& \hspace*{4cm} \times f_t(y)f_t(z)(1-F_t(z))^{n-2}dy dz,
\label{eq-13}
\end{align}
where the right-most equality follows from formula (\ref{ii-3}) with $i=2$. Next, we calculate the second sum on the right-hand side of equation (\ref{eq-12v}) and have that
\begin{align}
\sum_{j=3}^{n}& \mathbf{E}\big [ X_1 X_j | N(t)=n \big ]
\notag
\\
&=\sum_{j=3}^{n} \iiint\mathbf{E}\big [ X_1 X_j |T_{1}=y ,T_{j-1}=w ,T_{j}=z  \big ]
dF_{1,j-1,j|t,n}(y,w,z)
\notag
\\
&=\sum_{j=3}^{n} {n! \over (j-3)!(n-j)!} \int_{0}^{t}\int_{0}^{z}\int_{0}^{w}\mathbf{E}\big [ X_1 X_j |T_{1}=y ,T_{j-1}=w ,T_{j}=z  \big ]
\notag
\\
& \hspace*{2cm} \times f_t(y)f_t(w)f_t(z)(F_t(w)-F_t(y))^{j-3}(1-F_t(z))^{n-j}dydwdz,
\label{eq-14}
\end{align}
where we have used the fact that, for $3\le j \le n$, the density of $F_{1,j-1,j|t,n}$ is equal to (see, e.g., David and Nagaraja, 2003)
\begin{equation}
f_{1,j-1,j|t,n}(y,w,z)={n! \over (j-3)!(n-j)!} f_t(y)f_t(w)f_t(z)(F_t(w)-F_t(y))^{j-3}(1-F_t(z))^{n-j}
\label{ii-22}
\end{equation}
when $0\le y \le w \le z \le t$ and vanishes for all other $(y,w,z)$.
As to the third sum on the right-hand side of equation (\ref{eq-12v}), we have that
\begin{align}
\sum_{i=2}^{n-1}&\mathbf{E}\big [ X_i X_{i+1} | N(t)=n \big ]
\notag
\\
&=\sum_{i=2}^{n-1}\iiint\mathbf{E}\big [X_i X_{i+1} |T_{i-1}=y,T_{i}=w,T_{i+1}=z \big ]
dF_{i-1,i,i+1|t,n}(y,w,z)
\notag
\\
&=\sum_{i=2}^{n-1}{n! \over (i-2)!(n-i-1)!} \int_{0}^{t}\int_{0}^{z}\int_{0}^{w}\mathbf{E}\big [ X_i X_{i+1} | T_{i-1}=y ,T_{i}=w ,T_{i+1}=z \big ]
\notag
\\
& \hspace*{4cm} \times f_t(y)f_t(w)f_t(z)F_t(y)^{i-2}(1-F_t(z))^{n-i-1}dydwdz,
\label{eq-15}
\end{align}
where the latter follows from the fact that, for $2\le i \le n$, the density corresponding to $F_{i-1,i,i+1|t,n}$ is equal to (see, e.g., David and Nagaraja, 2003)
\begin{equation}
f_{i-1,i,i+1|t,n}(y,w,z)={n! \over (i-2)!(n-i-1)!} f_t(y)f_t(w)f_t(z)F_t(y)^{i-2}(1-F_t(z))^{n-i-1}
\label{ii-22v}
\end{equation}
when $0\le y \le w \le z \le t$ and vanishes for all other $(y,w,z)$.
Finally, we compute the fourth sum on the right-hand side of equation (\ref{eq-12v}) and have that
\begin{align}
&\sum_{i=2}^{n-2}\sum_{j=i+2}^{n}\mathbf{E}\big [ X_i X_j | N(t)=n \big ]
\notag
\\
&= \sum_{i=2}^{n-2}\sum_{j=i+2}^{n}\iiiint
\mathbf{E}\big [ X_i X_j | T_{i-1}=x, T_{i}=y ,T_{j-1}=w , T_{j}=z \big ]
dF_{i-1,i,j-1,j|t,n}(x, y,w,z)
\notag
\\
&= \sum_{i=2}^{n-2}\sum_{j=i+2}^{n}{n! \over (i-2)!(j-i-2)!(n-j)!}
\notag
\\
& \quad \times \int_{0}^{t}\int_{0}^{z}\int_{0}^{w}\int_{0}^{y}
\mathbf{E}\big [ X_i X_j | T_{i-1}=x, T_{i}=y ,T_{j-1}=w , T_{j}=z \big ]
\notag
\\
& \hspace*{1cm} \times f_t(x)f_t(y)f_t(w)f_t(z) F_t^{i-2}(x)(F_t(w)-F_t(y))^{j-i-2}(1-F_t(z))^{n-j} dxdydwdz,
\label{eq-16}
\end{align}
where the right-most equation holds because, for $2\le i <j \le n $ with $j-i\ge 2$, the density of $F_{i-1,i,j-1,j|t,n}$ is equal to (see, e.g., David and Nagaraja, 2003)
\begin{multline}
f_{i-1,i,j-1,j|t,n}(x, y,w,z)={n! \over (i-2)!(j-i-2)!(n-j)!} f_t(x)f_t(y)f_t(w)f_t(z)
\\
\times F_t^{i-2}(x)(F_t(w)-F_t(y))^{j-i-2}(1-F_t(z))^{n-j}
\label{ii-32v}
\end{multline}
within the region $0\le x \le y \le w \le z  \le t$ and vanishes for all other $(x, y,w,z)$.

With the above formulas, we arrive at equation (\ref{ii-42}) and finish the proof of Theorem \ref{theorem-02}.
\end{proof}

\begin{proof}[Proof of Theorem \ref{theorem-2}]
By assumption, each claim size $X_i$ depends only on the preceding inter-claim time $V_i=T_{i}-T_{i-1}$ (by definition, $T_0=0$) and thus the expectation $\mathbf{E} [ X_i | T_{i-1}=x, V_{i}=v]$ reduces to $\mathbf{E} [ X_i | V_i=v  ]$, which is $\Delta_i(v)$. Hence, equation \eqref{ii-6} becomes
\begin{multline}
\mathbf{E}[S(t)]
=\sum_{n=1}^{\infty } \pi_{t,n}^{L} \Bigg \{ \int_{0}^{t}\Delta_1(y)\bigg ( n{(t-y)^{n-1}\over t^n} \bigg )dy
\\
+ \sum_{i=2}^n \int_{0}^{t} \int_{0}^{t-x}\Delta_i(v)
\bigg ({n! \over (i-2)!(n-i)!} {x^{i-2}(t-x-v)^{n-i}\over t^n}\bigg )d vdx\Bigg \} .
\label{ii-7}
\end{multline}
Interchanging the order of integration on the right-hand side of equation (\ref{ii-7}) and then noticing that the resulting inner integral is the complete beta function, we obtain
\begin{align}
\mathbf{E}[S(t)]
&=\sum_{n=1}^{\infty } \pi_{t,n}^{L} \Bigg \{ \int_{0}^{t}\Delta_1(y)\bigg ( n{(t-y)^{n-1}\over t^n} \bigg )dy
\notag
\\
& \quad + \sum_{i=2}^n \int_{0}^{t} \Delta_i(v)\int_{0}^{t-v}
\bigg ({n! \over (i-2)!(n-i)!} {x^{i-2}(t-v-x)^{n-i}\over t^n}\bigg )dxd v\Bigg \}
\notag
\\
&=\sum_{n=1}^{\infty } \pi_{t,n}^{L} \Bigg \{ \int_{0}^{t}\Delta_1(y)\bigg ( n{(t-y)^{n-1}\over t^n} \bigg )dy
+ \sum_{i=2}^n \int_{0}^{t} \Delta_i(v)\bigg ( n{(t-v)^{n-1}\over t^n} \bigg )d v\Bigg \}
\notag
\\
&=\sum_{n=1}^{\infty } \pi_{t,n}^{L} \sum_{i=1}^n \int_{0}^{t} \Delta_i(v)\bigg ( n{(t-v)^{n-1}\over t^n} \bigg )d v.
\label{ii-7b}
\end{align}
Since $N(t)$ is  the mixed Poisson process with the structure cdf $L(\lambda)$, using formula (\ref{mixed-p}) on the right-hand side of equation (\ref{ii-7b}), we obtain
\begin{align}
\mathbf{E}[S(t)]
&=\sum_{n=1}^{\infty }\bigg\{
\bigg ( \int_{0}^{\infty}{(\lambda t)^n \over n!} e^{-\lambda t}dL(\lambda) \bigg )
\sum_{i=1}^{n} \int_{0}^t \Delta_i(v)\bigg ( n\,{(t-v)^{n-1}\over t^n}\bigg ) dv \bigg\}
\notag
\\
&=\int_{0}^{\infty} \int_{0}^t \lambda e^{-\lambda v} \sum_{n=1}^{\infty }\bigg\{
\bigg ( {(t-v)^{n-1}\lambda^{n-1}\over (n-1)!} e^{-\lambda (t-v)} \bigg )
\sum_{i=1}^{n} \Delta_i(v)\bigg\}\, dv\,dL(\lambda)
\notag
\\
&=\int_{0}^{\infty}  \int_{0}^t \lambda e^{-\lambda v}\sum_{n=0}^{\infty }\bigg\{
\bigg ( {(t-v)^{n}\lambda^{n}\over n!} e^{-\lambda (t-v)} \bigg )
\sum_{i=1}^{n+1} \Delta_i(v)\bigg\}\, dv\,dL(\lambda)
\notag
\\
&=\int_{0}^{\infty}  \int_{0}^t \lambda e^{-\lambda v}\sum_{n=0}^{\infty }\bigg\{
\mathbf{P}[N_{\lambda }(t-v)=n]
\sum_{i=1}^{n+1} \Delta_i(v)\bigg\}\, dv\,dL(\lambda)
\notag
\\
&= \int_{0}^{\infty}  \int_{0}^t \lambda e^{-\lambda v}
\mathbf{E}\big [Q\big (N_{\lambda }(t-v)+1|\,v \big )\big ]\, dv\,dL(\lambda).
\label{calcul-1}
\end{align}
This completes the proof of Theorem \ref{theorem-2}.
\end{proof}

\begin{proof}[Proof of Theorem \ref{theorem-1}]
Since $F_t(x)=x / t$ and thus $f_t(x)=1 / t$ for all $x\in [0,t]$, equation \eqref{ii-4} becomes
\begin{multline}
\mathbf{E}[S(t)]
=\sum_{n=1}^{\infty } \pi_{t,n}^{L} \Bigg \{ n\int_{0}^{t}\mathbf{E}\big [ X_1 | T_{1}=y \big ]{(t-y)^{n-1}\over t^n}dy
\\
 + \sum_{i=2}^n {n! \over (i-2)!(n-i)!} \int_{0}^{t} \int_{x}^{t}
\mathbf{E}\big [ X_i | T_{i-1}=x, T_{i}=y \big ]
{x^{i-2}(t-y)^{n-i}\over t^n}dydx\Bigg \} .
\label{ii-5}
\end{multline}
In the inner integral on the right-hand side of equation \eqref{ii-5}, we change the variable of integration $y$ into $v$ using the relationship $y=x+v$. Equation \eqref{ii-5} becomes
\begin{multline}
\mathbf{E}[S(t)]
=\sum_{n=1}^{\infty } \pi_{t,n}^{L} \Bigg \{ n\int_{0}^{t}\mathbf{E}\big [ X_1 | T_{1}=y \big ] {(t-y)^{n-1}\over t^n} dy
\\
 + \sum_{i=2}^n {n! \over (i-2)!(n-i)!}\int_{0}^{t} \int_{0}^{t-x}
\mathbf{E}\big [ X_i | T_{i-1}=x, T_{i}=x+v \big ]
 {x^{i-2}(t-x-v)^{n-i}\over t^n} dvdx\Bigg \} .
\label{ii-6b}
\end{multline}
Since $N(t)$ is  the mixed Poisson process with the structure cdf $L(\lambda )$, we apply formula (\ref{mixed-p}) on the right-hand side of equation (\ref{ii-6b}), rearrange terms, and arrive at equation (\ref{ii-6}). This concludes the proof of Theorem \ref{theorem-1}.
\end{proof}

\begin{proof}[Proof of Corollary \ref{cor-0}]
Since $\mathbf{E}[ X_i|V_i=v]=\mathbf{E} [ X_1| V_1=v]$ for all $i\ge 1$, we have that  $Q(n|\,v)=n\mathbf{E} [ X_1| V_1=v]$. Since $\mathbf{E} [N_{\lambda }(t-v)+1]=\lambda(t-v)+1$, equation \eqref{th-0aa} completes the proof of Corollary \ref{cor-0}.
\end{proof}

\begin{proof}[Proof of Corollary \ref{cor-1}]
Equation \eqref{eq-12} follows from \eqref{th-0bb} by noting that the homogenous Poisson process is a mixed Poisson process with the structure cdf $L(z)=\mathbf{1}\{\lambda \le z\}$. This completes the proof of Corollary \ref{cor-1}.
\end{proof}

\begin{proof}[Proof of Corollary \ref{cor-4}]
Under assumption \eqref{eqcd}, we have that
\[
\mathbf{E}[X_i|V_i=v]=\left(1-e^{-\beta v} \right)\mathbf{E}[Y_{\ell}]+ e^{-\beta v}\mathbf{E}[Y_s].
\]
Applying this formula on the right-hand side of equation \eqref{th-0bb}, we have that the inner integral there is equal to
\begin{equation}
\mathbf{E}[Y_{\ell}]\int_{0}^t \left(1-e^{-\beta v} \right)\lambda e^{-\lambda v} \big ( \lambda(t-v)+1 \big ) dv
+\mathbf{E}[Y_s]\int_{0}^t e^{-\beta v} \lambda e^{-\lambda v} \big ( \lambda(t-v)+1 \big ) dv .
\label{exp-boudr-2}
\end{equation}
The right-most integral of (\ref{exp-boudr-2}) is equal to $\mathfrak{A}(t,\lambda,\beta )$, and the left-most one is equal to $\lambda t-\mathfrak{A}(t,\lambda,\beta )$. This completes the proof of Corollary \ref{cor-4}.
\end{proof}

\begin{proof}[Proof of Theorem \ref{theorem-12bb}]
Since the claim sizes $X_i$ depend only on the preceding inter-claim times $V_i$ and the conditional variables $X_i|V_i$ are independent, quantities $A_{t,n}^{L}$ and $B_{t,n}^{L}$ in Theorem \ref{th-new-12} reduce to the following ones:
\begin{multline}\label{quantity-a}
A_{t,n}^{L}= n \int_{0}^{t}\Theta_1(y){(t-y)^{n-1}\over t^n} dy
\\
+ \sum_{i=2}^n {n! \over (i-2)!(n-i)!}\int_{0}^{t} \int_{0}^{t-x}
\Theta_i(v) {x^{i-2}(t-x-v)^{n-i}\over t^n} dv dx
\end{multline}
and
\begin{align}
B_{t,n}^{L}&=n(n-1) \int_{0}^{t} \int_{0}^{t-y} \Delta_1(y)\Delta_2(v){(t-y-v)^{n-2} \over t^n}dv dy
\notag
\\
&\quad + \sum_{j=3}^{n} {n! \over (j-3)!(n-j)!} \int_{0}^{t}\int_{y}^{t} \int_{0}^{t-w} \Delta_1(y)\Delta_j(v)
\notag
\\
& \hspace*{45mm} \times {(w-y)^{j-3}(t-w-v)^{n-j}\over t^n}dvdwdy
\notag
\\
&\quad + \sum_{i=2}^{n-1}{n! \over (i-2)!(n-i-1)!} \int_{0}^{t}\int_{0}^{t-y} \int_{0}^{t-y-u}\Delta_i(u)\Delta_{i+1}(v)
\notag
\\
& \hspace*{45mm} \times {y^{i-2}(t-y-u-v)^{n-i-1}\over t^n}dvdudy
\notag
\\
&\quad + \sum_{i=2}^{n-2}\sum_{j=i+2}^{n}{n! \over (i-2)!(j-i-2)!(n-j)!}
\int_{0}^{t}\int_{0}^{t-x} \int_{x+u}^{t}\int_{0}^{t-w}
\Delta_i(u)\Delta_j(v)
\notag
\\
& \hspace*{45mm} \times { x^{i-2}(w-x-u)^{j-i-2}(t-w-v)^{n-j} \over t^n}dvdwdudx.
\label{ii-42bb}
\end{align}
We now calculate the sum on the right-hand side of equation \eqref{quantity-a}. Interchanging the order of integration, we obtain that
\[
\sum_{i=2}^n {n! \over (i-2)!(n-i)!}\int_{0}^{t} \int_{0}^{t-x}
\Theta_i(v) {x^{i-2}(t-x-v)^{n-i}\over t^n} dv dx
= \sum_{i=2}^n n \int_{0}^{t}\Theta_i(v){(t-v)^{n-1}\over t^n} dv.
\]
This implies that
\begin{align}
A_{t,n}^{L}= \sum_{i=1}^n n \int_{0}^{t}\Theta_i(v){(t-v)^{n-1}\over t^n} dv.
\label{ii-42bb1}
\end{align}
We next calculate the integrals that make up the quantity $B_{t,n}^{L}$. First, interchanging the order of integration and utilizing the definition and properties of the beta function (e.g., Abramowitz and Stegun, 1972, formula 6.2.2), we obtain
\begin{multline}
\sum_{j=3}^{n} {n! \over (j-3)!(n-j)!} \int_{0}^{t}\int_{y}^{t} \int_{0}^{t-w} \Delta_1(y)\Delta_j(v) {(w-y)^{j-3}(t-w-v)^{n-j}\over t^n}dvdwdy
\\
=\sum_{j=3}^{n} n(n-1) \int_{0}^{t} \int_{0}^{t-y} \Delta_1(y)\Delta_j(v){(t-y-v)^{n-2} \over t^n}dv dy .
\label{ii-42bb2}
\end{multline}
In a similar fashion we show that
\begin{multline}
\sum_{i=2}^{n-1}{n! \over (i-2)!(n-i-1)!} \int_{0}^{t}\int_{0}^{t-y} \int_{0}^{t-y-u}\Delta_i(u)\Delta_{i+1}(v)
{y^{i-2}(t-y-u-v)^{n-i-1}\over t^n}dvdudy
\\
=\sum_{i=2}^{n-1} n(n-1)\int_{0}^{t} \int_{0}^{t-u} \Delta_i(u)\Delta_{i+1}(v){(t-u-v)^{n-2} \over t^n}dv du .
\label{ii-42bb3}
\end{multline}
Finally, we calculate the double sum on the right-hand side of equation \eqref{ii-42bb}. Replacing $w$ by $w-x-u$, we obtain
\begin{align}
&\sum_{i=2}^{n-2}\sum_{j=i+2}^{n}{n! \over (i-2)!(j-i-2)!(n-j)!}
\int_{0}^{t}\int_{0}^{t-x} \int_{x+u}^{t}\int_{0}^{t-w}
\Delta_i(u)\Delta_j(v)
\notag
\\
& \hspace*{65mm} \times { x^{i-2}(w-x-u)^{j-i-2}(t-w-v)^{n-j} \over t^n}dvdwdudx
\notag
\\
&=\sum_{i=2}^{n-2}\sum_{j=i+2}^{n}{n! \over (i-2)!(j-i-2)!(n-j)!}
\int_{0}^{t}\int_{0}^{t-x} \int_{0}^{t-x-u}\int_{0}^{t-x-u-w}
\Delta_i(u)\Delta_j(v)
\notag
\\
& \hspace*{65mm} \times { x^{i-2}w^{j-i-2}(t-x-u-v-w)^{n-j} \over t^n}dvdwdudx
\notag
\\
&=\sum_{i=2}^{n-2}\sum_{j=i+2}^{n}{n! \over (i-2)!(j-i-2)!(n-j)!}
\int_{0}^{t}\int_{0}^{t-u} \Delta_i(u)\Delta_j(v)\int_{0}^{t-u-v}\int_{0}^{t-u-v-w}
\notag
\\
& \hspace*{65mm} \times { x^{i-2}w^{j-i-2}(t-x-u-v-w)^{n-j} \over t^n}dxdwdvdu
\notag
\\
&=\sum_{i=2}^{n-2}\sum_{j=i+2}^{n}{n! \over (j-i-2)!(n-j+i-1)!}
\int_{0}^{t}\int_{0}^{t-u} \Delta_i(u)\Delta_j(v)
\notag
\\
& \hspace*{65mm} \times \int_{0}^{t-u-v}{ w^{j-i-2}(t-u-v-w)^{n-j+i-1} \over t^n}dwdvdu
\notag
\\
&=\sum_{i=2}^{n-2}\sum_{j=i+2}^{n}n(n-1)
\int_{0}^{t}\int_{0}^{t-u} \Delta_i(u)\Delta_j(v)
{ (t-u-v)^{n-2} \over t^n}dvdu .
\label{ii-42bb4}
\end{align}
Using equations \eqref{ii-42bb2}--\eqref{ii-42bb4} on the right-hand side of \eqref{ii-42bb}, we obtain
\begin{align}
B_{t,n}^{L}
&= n(n-1) \int_{0}^{t} \int_{0}^{t-y} \Delta_1(y)\Delta_2(v){(t-y-v)^{n-2} \over t^n}dv dy
\notag
\\
&\qquad + \sum_{j=3}^{n} n(n-1) \int_{0}^{t} \int_{0}^{t-y} \Delta_1(y)\Delta_j(v){(t-y-v)^{n-2} \over t^n}dv dy
\notag
\\
&\qquad + \sum_{i=2}^{n-1} n(n-1)\int_{0}^{t} \int_{0}^{t-u} \Delta_i(u)\Delta_{i+1}(v){(t-u-v)^{n-2} \over t^n}dv du
\notag
\\
&\qquad + \sum_{i=2}^{n-2}\sum_{j=i+2}^{n}
n(n-1)\int_{0}^{t}\int_{0}^{t-u} \Delta_i(u)\Delta_j(v)
{ (t-u-v)^{n-2} \over t^n}dvdu
\notag
\\
&= n(n-1) \int_{0}^{t} \int_{0}^{t-y} \sum_{1\le i < j \le n}\Delta_i(y)\Delta_j(v){(t-y-v)^{n-2} \over t^n}dv dy.
\label{ii-42bb6}
\end{align}
We have from  (\ref{ii-42a}), (\ref{ii-42bb1}), and (\ref{ii-42bb6}) that
\begin{align}
\mathbf{E}[S^2(t)]
&=\int_{0}^{\infty}\int_{0}^{t}\bigg \{\sum_{n=1}^{\infty } {(\lambda t)^n \over n!} e^{-\lambda t} \Theta(n|\,v)\,n {(t-v)^{n-1}\over t^n} \bigg \}dvdL(\lambda)
\notag
\\
&\quad + \int_{0}^{\infty}\int_{0}^{t} \int_{0}^{t-y}\bigg \{ \sum_{n=1}^{\infty } {(\lambda t)^n \over n!} e^{-\lambda t} \Upsilon(n|\,y,v)\, n(n-1){(t-y-v)^{n-2} \over t^n}\bigg \}dv dy dL(\lambda)
\notag
\\
&=\int_{0}^{\infty}  \int_{0}^t \lambda e^{-\lambda v}
\mathbf{E}\big [\Theta \big (N_{\lambda }(t-v)+1|\,v \big )\big ]\, dv\,dL(\lambda)
\notag
\\
&\quad + \int_{0}^{\infty}\int_{0}^{t} \int_{0}^{t-y}\bigg \{ \sum_{n=2}^{\infty } {\lambda^n (t-y-v)^{n-2}\over (n-2)!} e^{-\lambda t} \Upsilon(n|\,y,v)\bigg \}dv dy dL(\lambda)  .
\label{ii-42ccc}
\end{align}
In the above calculations, we have obtained the expectation $\mathbf{E} [\Theta (N_{\lambda }(t-v)+1|\,v )]$ inside the double integral using analogous considerations as those in equation (\ref{th-0aa}). Similarly, we calculate the sum inside the triple integral on the right-hand side of equation \eqref{ii-42ccc}. This leads us to equation (\ref{ii-42ddd}). The proof of Theorem \ref{theorem-12bb} is finished.
\end{proof}

\begin{proof}[Proof of Corollary \ref{cor-12bb}]
Since $\Theta_1(v)=\Theta_i(v)$ and $\Delta_1(v)=\Delta_i(v)$ for all $i\ge 1$, we have
\begin{align*}
\mathbf{E} [\Theta(N_{\lambda }(t-v)+1|v)] = & \mathbf{E}[N_\lambda(t-v)+1]\Theta_1(v) \\
= & (\lambda (t-v)+1)\Theta_1(v)
\end{align*}
and
\begin{align*}
\mathbf{E} [\Upsilon(N_\lambda(t-v)+2|y,v)] = & \mathbf{E}[(N_{\lambda }(t-y-v)+2)(N_{\lambda }(t-y-v)+1)]\Delta_1(y)\Delta_1(v) \\
= & \left(\big (\lambda (t-y-v)+2 \big )^2-2\right)\Delta_1(y)\Delta_1(v).
\end{align*}
Using these formulas on the right-hand side of equation \eqref{ii-42ddd}, we complete the proof of Corollary \ref{cor-12bb}.
\end{proof}

\begin{proof}[Proof of Corollary \ref{cor-4g}]
Using formula \eqref{eqcd}, we obtain from equation \eqref{ii-42ddd2} that
\begin{multline}
\mathbf{E}[S^2(t)]
=\int_{0}^{\infty}  \int_{0}^t \lambda e^{-\lambda v} \Big ( \left(1-e^{-\beta v} \right)\mathbf{E}[Y_{\ell}^2]+ e^{-\beta v}\mathbf{E}[Y_s^2] \Big )
\big (\lambda (t-v)+1\big ) dv\,dL(\lambda)
\\
+ \int_{0}^{\infty}\mathfrak{B}(t,\lambda ) dL(\lambda) ,
\label{ii-42ddd3a}
\end{multline}
where we have used the notation
\begin{multline*}
\mathfrak{B}(t,\lambda )=\int_{0}^{t} \int_{0}^{t-y}\lambda^2 e^{-\lambda (y+v)}
\Big ( \left(1-e^{-\beta y} \right)\mathbf{E}[Y_{\ell}]+ e^{-\beta y}\mathbf{E}[Y_s] \Big )
\\
\times \Big ( \left(1-e^{-\beta v} \right)\mathbf{E}[Y_{\ell}]+ e^{-\beta v}\mathbf{E}[Y_s] \Big ) \left(\big (\lambda (t-y-v)+2 \big )^2-2\right)dv dy .
\end{multline*}
We already know from the proof of Corollary \ref{cor-4} that
\begin{multline}\label{formula-1aaa}
\int_{0}^t \lambda e^{-\lambda v} \Big ( \left(1-e^{-\beta v} \right)\mathbf{E}[Y_{\ell}^2]+ e^{-\beta v}\mathbf{E}[Y_s^2] \Big )
\big (\lambda (t-v)+1\big ) dv
\\
 =  \mathbf{E}[Y_{\ell}^2] \big(\lambda t-\mathfrak{A}(t,\lambda,\beta ) \big)
+\mathbf{E}[Y_s^2] \mathfrak{A}(t,\lambda,\beta ) .
\end{multline}
This givens the first two summands on the right-hand side of equation (\ref{ii-42ddd4}). It remains to calculate the right-most integral of \eqref{ii-42ddd3a}. For this, we rewrite $\mathfrak{B}(t,\lambda )$ as follows:
\begin{align*}
\mathfrak{B}(t,\lambda )=& \big (\mathbf{E}[Y_{\ell}]\big )^2\lambda^2\int_{0}^{t} e^{-\lambda y}\left(1-e^{-\beta y} \right)\int_{0}^{t-y} e^{-\lambda v} \left(1-e^{-\beta v} \right)
\left(\big (\lambda (t-y-v)+2 \big )^2-2\right)dv dy
\\
&+ 2 \mathbf{E}[Y_{\ell}]\mathbf{E}[Y_s]\lambda^2\int_{0}^{t} e^{-\lambda y}\left(1-e^{-\beta y} \right)\int_{0}^{t-y} e^{-\lambda v}
e^{-\beta v}
\left(\big (\lambda (t-y-v)+2 \big )^2-2\right)dv dy
\\
&+ \big ( \mathbf{E}[Y_s]\big )^2 \lambda^2\int_{0}^{t} e^{-\lambda y}e^{-\beta y}\int_{0}^{t-y} e^{-\lambda v}
  e^{-\beta v}\left(\big (\lambda (t-y-v)+2 \big )^2-2\right)dv dy.
\end{align*}
To simplify the presentation, we next introduce a general notation encompassing the $\mathfrak{B}$-quantities in formulas (\ref{f-1}), (\ref{f-3}), and (\ref{f-2}):
\[
\mathfrak{B}(t,\lambda,\Theta,\Delta )=\lambda^2\int_{0}^{t} e^{-(\lambda+\Theta) y}\int_{0}^{t-y} e^{-(\lambda+\Delta) v}\left(\big (\lambda (t-y-v)+2 \big )^2-2\right)dv dy.
\]
Note the symmetry: $\mathfrak{B}(t,\lambda,\Theta,\Delta )=\mathfrak{B}(t,\lambda,\Delta,\Theta )$. After a somewhat tedious checking of formulas (\ref{f-1}) and (\ref{f-3}), we obtain that
\begin{multline*}
\mathfrak{B}(t,\lambda )= \big (\mathbf{E}[Y_{\ell}]\big )^2
\Big ( \mathfrak{B}(t,\lambda,0,0 )-2\mathfrak{B}(t,\lambda,0,\beta )+\mathfrak{B}(t,\lambda,\beta,\beta ) \Big )
\\
+ 2 \mathbf{E}[Y_{\ell}]\mathbf{E}[Y_s]\Big ( \mathfrak{B}(t,\lambda,0,\beta )-\mathfrak{B}(t,\lambda,\beta,\beta ) \Big )
+ \big ( \mathbf{E}[Y_s]\big )^2 \mathfrak{B}(t,\lambda,\beta,\beta ).
\end{multline*}
This finishes the proof of Corollary \ref{cor-4g}.
\end{proof}

\end{document}